\documentclass[11pt]{amsart}
\usepackage{amssymb}
\usepackage{amsmath}
\usepackage{amsthm}
\usepackage[textwidth=16cm, textheight=24cm, marginratio=1:1]{geometry}

\usepackage[english]{babel}
\usepackage[utf8]{inputenc}
\usepackage[T1]{fontenc}
\usepackage{todonotes}
\usepackage{xcolor}
\usepackage{booktabs}

\usepackage[backref=page]{hyperref}
\hypersetup{
    colorlinks,
    linkcolor={red!80!black},
    citecolor={blue!80!black},
    urlcolor={blue!80!black}
}
\usepackage{setspace}
\usetikzlibrary{matrix,arrows, cd, calc}  
\numberwithin{equation}{section}
\newtheorem{theorem}{Theorem}[section]
\newtheorem{corollary}[theorem]{Corollary}
\newtheorem{lemma}[theorem]{Lemma}

\newtheorem{proposition}[theorem]{Proposition}
\theoremstyle{definition}
\newtheorem{definition}[theorem]{Definition}
\newtheorem{remark}[theorem]{Remark}
\newtheorem{example}[theorem]{Example}
\newtheorem{conjecture}[theorem]{Conjecture}
\newtheorem{question}[theorem]{Question}
\newtheorem{notation}[theorem]{Notation}

\DeclareMathOperator{\Spec}{Spec}

\DeclareMathOperator{\Frac}{Frac}

\newcommand{\into}{\hookrightarrow}

\newcommand{\kk}{\Bbbk}%
\DeclareMathOperator{\Speck}{Spec_{\kk}}
\newcommand{\spann}[1]{\left\langle #1 \right\rangle}

\newcommand{\kkt}{\kk[\![t]\!]}
\newcommand{\ttfrac}[1]{{#1}(\!(t)\!)}
\newcommand{\ttpuis}[1]{{#1}\{\!\{t\}\!\}}
\newcommand{\kktfrac}{\kk(\!(t)\!)}
\newcommand{\kktpuis}{\kk\{\!\{t\}\!\}}
\newcommand{\Ct}[1]{\mathbb{C}_{#1}\{t\}}
\newcommand{\mm}{\mathfrak{m}}%
\newcommand{\pp}{\mathfrak{p}}%
\newcommand{\qq}{\mathfrak{q}}%
\newcommand{\kkxy}{\kk\langle x, y \rangle}%
\newcommand{\kkxyhat}{\kk\langle\!\langle x, y \rangle\!\rangle}%
\newcommand{\cM}{\mathcal{M}}%

\DeclareMathOperator{\HH}{HH}

\begin{document}
\def\F{{\mathbb F}}
\title{On a conjecture by Michael Wemyss regarding the calculation of GV invariants}
\author{Joachim Jelisiejew}
\address[Joachim Jelisiejew]{University of Warsaw, Faculty of Mathematics, Informatics and Mechanics,
 Stefana Banacha~2, 02-097 Warsaw, Poland.}
\email{j.jelisiejew@uw.edu.pl}
\author{Agata Smoktunowicz}
\address[Agata Smoktunowicz]{University of Edinburgh, Peter Guthrie Tait Road, King's Buildings, Edinburgh, United Kingdom.}
\email{A.Smoktunowicz@ed.ac.uk}
\date{\today{}}
\begin{abstract}
Contraction algebras are noncommutative algebras introduced by Donovan and Wemyss to classify
of $3$-dimensional flops. Wemyss conjectures that contraction algebras can be deformed to a single semisimple algebra.
This gives an intrinsic method of calculating Gopakumar-Vafa invariants of the flop.

The main result is a proof of Wemyss' conjecture for types A and D.
In the course of the proof, we recall and introduce new techniques for constructing flat deformations of
associative algebras and compare various notions of deformations. We also put forward
two conjectures which hint towards a deeper theory.
\end{abstract}
\maketitle

\section{Introduction}

    In this article, we investigate an interesting class of noncommutative
    algebras, so-called contraction algebras of Donovan-Wemyss. The results can be viewed from
    two completely different angles, which can be very crudely summarised as follows
    \begin{enumerate}
        \item From the point of view of algebraic geometry, contraction algebras $A$ appear
            as a classifying object of flops in dimension three (see~\S\ref{ssec:geometry} below). The curve which is flopped yields
            Gopakumar-Vafa invariants. The overarching question is
            \begin{question}\label{ref:GVabstract}
                Can the GV invariants be recovered solely from the abstract algebra $A$?
            \end{question}
            The question can be restated as
            \begin{question}\label{ref:GVconcrete}
                Does $A$ deform to a unique semisimple algebra
                \[\kk^{\times n_1}\times \mathbb{M}_2(\kk)^{\times n_2}\times
                    \mathbb{M}_3(\kk)^{\times n_3}\times \ldots
                \] where $n_1, n_2, n_3$
                are the GV invariants given by geometry?
            \end{question}
        \item From the point of view of noncommutative algebra, contraction algebras $A$ can be viewed as
            ``given from above'': they are finite-dimensional algebras with interesting deformations.
            Hence, they fit into the long line of research into deformations of such algebras (discussed in~\S\ref{ssec:algebra} below)
            and, much more importantly, suggest new lines of research in this field, see Conjecture~\ref{conj:unobstructedness} below.
    \end{enumerate}
    We want to keep the article accessible to both communities and so we
    restrict to elementary methods whenever possible.

    Contraction algebras are defined formally in Definition~\ref{ref:contractionTypes:def} below.
    A representative example is the algebra $D_{n,m}$ presented as
    \begin{equation}\label{eq:Dnm}
        \frac{\kkxy}{\left( xy+yx,\ x^{2n-1} + x^{2m-2} + y^2,\ x^{2n+2m-3} \right)}
    \end{equation}
    where $n\geq m\geq 2$.

    Our main result answers Question~\ref{ref:GVconcrete} affirmatively, for types A and D.
    \begin{theorem}[GV invariants are determined by contraction algebras,
        Theorem~\ref{ref:obstructionsMain:thm},
        Theorem~\ref{ref:existence:thm}]\label{ref:GVinvariants:thm}
        Contraction algebras of type $A$ and $D$ deform to a single semisimple algebra, as listed in Theorem~\ref{ref:obstructionsMain:thm}.
    \end{theorem}
    This is quite unexpected algebraically, since ``most'' commutative algebras do not deform to a single semisimple one, and quite some do not
    deform to any such, see the sources in~\S\ref{ssec:algebra}. Thus, one can hope that contraction algebras form a new, interesting
    and well-behaved class of finite-dimensional noncommutative algebras.

    Pushing the idea of Theorem~\ref{ref:GVinvariants:thm} further, we observe that contraction algebras
    behave much more regularly than expected even by Theorem~\ref{ref:GVinvariants:thm}, which hints at a deeper theory.
    We formulate this precisely in the following conjecture, which is based also on discussions with Michael Wemyss.
    \begin{conjecture}\label{conj:unobstructedness}
        Let $A$ be a contraction algebra and $d = \dim_{\kk} A$. Then
        \begin{enumerate}
            \item\label{it:unobstructedness} a point $[A]$ corresponding to $A$
                in the moduli space $\cM_d$ of associative unital algebras
                (see~\S\ref{sec:deformations}) is smooth,
                in other words, $A$ is unobstructed,
            \item the dimension of the centre is constant in an open neighbourhood of $[A]$ and the
                associated map $A\mapsto Z(A)$ is smooth.
        \end{enumerate}
    \end{conjecture}

    Both parts of the conjecture are rather surprising.
    First, the centre of an algebra usually behaves poorly in families, even in best behaved cases:
    for the quantisation deformation
    \[
        \frac{\mathbb{C}\langle x, y \rangle[t]}{(xy - yx - t)}
    \]
    the centre for $t = 0$ is
    much larger than for $t = \lambda\in \mathbb{C}\setminus \{0\}$. A more trivial example is that every (unital associative) rank $d$ algebra $A$
    is a deformation of a trivial algebra $A_0 = \kk[x_1, \ldots ,x_{d-1}]/(x_1, \ldots ,x_{d-1})^2$, which is commutative.

    Second, the deformations of $Z(A)$ are much more understood, since $Z(A)$ is commutative. An unobstructedness result would mean that
    to understand deformations of $A$, we can reduce to the commutative case, where much more tools are available~\S\ref{ssec:algebra}.
    
    We verified Conjecture~\ref{conj:unobstructedness}\eqref{it:unobstructedness} for small
    values of $d$ using GAP's~\cite{GAP4} package QPA~\cite{QPA}. Importantly, the conjecture
    does \emph{not} follow from vanishing of appropriate Hochschild cohomology group $\HH^3(A, A)$.
    This group does not vanish already on our small examples. Instead, we use Theorem~\ref{ref:GVinvariants:thm}
    to provide a lower bound for the dimension of $\cM_d$ near $[A]$ and show that the smoothness
    of $[A]$ is equivalent to $\dim_{\kk} \HH^1(A, A) = \dim_{\kk} \HH^2(A, A)$, which we verified directly.

    A possible approach to Conjecture~\ref{conj:unobstructedness} is given by the theory of noncommutative
    regular normalising sequences also known as noncommutative complete intersections~\cite{Cassidy_Vancliff, Levasseur, Shelton_Tingey}.
    For example, the contraction algebra $D_{n,m}$ from~\eqref{eq:Dnm} can be viewed as a quotient of
    \[
        S = \frac{\kkxy}{(xy+yx)}
    \]
    by the relations $x^{2n-1}y$, $x^{2n-1} + x^{2m-2} + y^2$ and then completed by dividing by $x^{2n+2m-3}$.
    The algebra $S$ is AS-regular, the sequence $x^{2n-1}y$, $x^{2n-1} + x^{2m-2} + y^2$ is normalising
    and so the quotient resembles a complete intersection. There are some technical issues to resolve using this approach.
    First, completion by passing to $\kkxyhat$ is necessary, second, the theory of noncommutative complete intersections 
    seems to exist mostly in the graded setting, third, a result stating that such complete intersections are unobstructed
    is currently missing, as least according to our knowledge. Viewed from commutative algebra perspective,
    these obstacles should be surmountable, as the commutative theory followed a similar arc and the final results are
    obtained, see for example~\cite[\S2.9]{HarDeform}.

\subsection{Deformations of finite-dimensional algebras}\label{ssec:algebra}

The study of deformations of associative finite-dimensional is a classical and difficult subject.
Some surveys are~\cite{Gerstenhaber_Schack__survey_general, Schedler, Fox}.
Its cohomological foundations were build
by Gerstenhaber~\cite{G1, G2, G3, G4, S}, also later with Schack~\cite{Gerstenhaber_Schack__physical, Gerstenhaber_Schack__survey_general} and others~\cite{Gerstenhaber_Giaquinto__Compatible}. The moduli space of structure constants was used early on~\cite{Artin__Azumaya, Procesi, Voigt, Gabriel, Flanigan, Flanigan2} and is sometimes
crucial to use basic geometric structures, such as tangent spaces~\cite{LeBruyn_Reichstein} and orbit dimensions~\cite{GP}. It naturally appears in applications, such as to higher-rank vector bundles~\cite{LeBruyn_Reichstein} or to group algebras, related to Donald-Flanigan conjecture. The geometry of this space is known for small dimensions, up to eight, see~\cite{Happel, Mazzola, DanaPicard_Schaps_up_to_six, DanaPicard_Schaps_up_to_seven}, but unknown in general and quite pathological, with nonreduced components in positive characteristic~\cite{DanaPicard_Schaps_nonreduced, Gerstenhaber_Schack__relative}. It is natural and somewhat easier to linearise the problem and classify left modules~\cite{Modules_over_Weyl, Modules_over_free}. The theory is much related to tilting~\cite{Happel_HH, King_moduli}. The field remains active, see~\cite{Amsalem_Ginosar, D, FO, HMLR, MRRS, Shelter}. Much more is known if one assumes commutativity, as in this case the space is directly related to the Hilbert scheme of points~\cite{poonen_moduli_space, JJ, JJ3, JJ4, JJ5}. Deformations are also appearing in Kontsevich's deformation-quantisation, see for example~\cite{HV}, which we do not discuss further.

In the context of contraction algebras, strongly flat deformations are introduced in~\cite{Wemyss2} and small cases of Theorem~\ref{ref:GVinvariants:thm} are proven in~\cite{DDS}.

We recall much of the theory in Section~\ref{sec:deformations} in a way, hopefully, accessible to both target audiences. From the algebraic-geometry point
of view the theory is likely very standard, with the possible exception of
Beauville-Laszlo-type result (Theorem~\ref{ref:BeauvilleLaszlo:thm}) which
allows us to pass from deformations over $\kkt$ to the ones over $\kk[t]$.
As far as we know, this result does not follow formally from Beauville-Laszlo's
result (due to the technical fact that noncommutative algebras are not
equivalent to cyclic modules), yet it does follow from the proof.

\subsection{Algebraic geometry context}\label{ssec:geometry}

Contraction algebras are introduced by Donovan and Wemyss in the context of Wemyss' successful programme on classification of $3$-dimensional flops. In particular, with Iyama~\cite{IW} Wemyss introduced maximal modification algebras, and later, with Donovan~\cite{DW}, refined them to contraction algebras. Brown and Wemyss recently gave a purely algebraic description of these algebras in the local case, as completed noncommutative Jacobi algebras~\cite{GV}. The story is beautifully summarised in Wemyss' ICM 2026 talk materials. See also \cite{BD, HT, 10, K1, K2, Toda, Zhang} for deformation-theoretic part of this beautiful programme.

The curve which is flopped yields Gopakumar-Vafa (GV) invariants. They were related by Toda~\cite{Toda} to deformations of its thickenings.
Toda also observed that deforming the flop to simpler flops yields a family of contraction algebras, but it is not clear whether the family is flat.
Therefore, it is very interesting whether the GV invariants can be read off directly from the contraction algebra itself.

 \subsection{Outline of the paper}

 After the introduction, the paper is divided into two parts. In Section~\ref{sec:deformations} we recall the general theory
 of deformations and prove Beauville-Laszlo descent (Theorem~\ref{ref:BeauvilleLaszlo:thm}), culminating in Theorem~\ref{J1} which
 asserts that all introduced definitions of deformations agree and yields
 transitivity of deformations. In Section~\ref{sec:contraction} we prove
 Theorem~\ref{ref:GVinvariants:thm}, which follows from two results of very
 different flavour. Theorem~\ref{ref:obstructionsMain:thm} provides
 obstructions to deformations to semisimple algebras except for one type, which
 Theorem~\ref{ref:existence:thm} provides the desired deformation.

 \section{Deformations of finite-dimensional associative algebras}\label{sec:deformations}

 All of the algebras we consider are associative and have an identity element. 
In this section we review the various existing notions of deformations, as a reference especially
for the noncommutative algebraic audience. Most results will be familiar for algebraic geometers.
A notable exception is the Beauville-Laszlo descent for finite-dimensional algebras.
The main result is Theorem~\ref{J1}, which shows that many notions of deformations agree.

We work over an algebraically closed field $\kk$ of characteristic zero. By $\kkt$ we denote the formal power series and by $\kktfrac$ its field of fractions.
By $\kk[t^{\pm 1}]$ we denote the ring $\kk[t, t^{-1}]$.
By
\begin{equation}\label{eq:Puiseux}
 \kktpuis = \bigcup_{n\geq 0} \kk\left(\!\left(t^{1/n}\right)\!\right)
\end{equation}
we denote the field of Puiseux series. This field is the algebraic closure of
$\kktfrac$, see, for
example~\cite[Theorem~2.1.5]{Maclagan_Sturmfels__Tropical_Geometry}.
Sometimes, when considering convergent power series, we will reduce to $\kk = \mathbb{C}$. In this case, we let
$\Ct{r}\subseteq \mathbb{C}[\![t]\!]$ to be the subring of power series convergent in a fixed radius $r>0$ around zero.

Throughout the article, $A$, $B$, $C$ are associative, unital, possibly non-commutative, finite-dimensional $\kk$-algebras.
For such an algebra $A$, we denote by $A[t]\simeq A\otimes_{\kk} \kk[t]$ the
ring of polynomials with coefficients in $A$, and by $A[\![t]\!] = A\otimes_{\kk}
\kkt$ the ring of formal power series with coefficients in $A$. When $\kk = \mathbb{C}$, we
denote by $A\{t\} := A\otimes_{\mathbb{C}} \Ct{r}$ the ring of power series convergent in radius $r>0$ with coefficients in $A$.
We omit the $r$ from notation $A\{t\}$, as the precise value of $r$ will be of little importance for us.

 We recall the definition of the formal deformation of $A$ from \cite{S}.
 
 \begin{definition}\label{SWbook} A \emph{formal deformation} $(A_{t}, \circ, +)$ of $A$ is an associative $\kkt$-bilinear multiplication $\circ $ on the $\kkt$-module $A[\![t]\!]$. We assume that the image of the identity element in $A$ acts as the identity element in $A[\![t]\!]$.
     \begin{enumerate}
         \item Let $\kk = \mathbb{C}$. The formal deformation $(A_{t}, \circ, +)$ is {\em convergent (in radius $r$)}
             if for every $a,b\in A$, the element $a\circ b$ lies in $A\{t\}$.
         \item The deformation $(A_{t}, \circ, +)$ is of {\em polynomial type}
             if for every $a,b\in A$, the element $a\circ b$ lies in $A[t]$.  
     \end{enumerate}
\end{definition}

\begin{remark}\label{ref:fibre}
For $\kk = \mathbb{C}$, let $A_t$ be a convergent formal deformation of $A$ and let $a_{1}, \ldots , a_{d}\in A$ be a basis of $A$ as a linear vector space over field $\mathbb C$. Note that every element of $A\{t\}$ can be written in the form
\[\sum_{i=1}^{n}a_{i}f_{i}(t),\]
 for some $f_{i}\in \mathbb C\{t\}$. We then have 
\[a_{i}\circ a_{j}=\sum_{1\leq i,j,k\leq n}g_{i,j}^{k}(t)a_{k}.\] 
By assumption, there is a non-empty neighbourhood of $0\in \mathbb{C}$ on which all $g_{i,j}^k$ are convergent. For $s$ in this neighbourhood
we can substitute $t :=s$ and obtain a $\mathbb{C}$-algebra with multiplication
\[a_{i}\circ a_{j}=\sum_{1\leq i,j,k\leq n}g_{i,j}^{k}(s)a_{k}.\] 
We denote the resulting $\mathbb{C}$-algebra by $A_{t=s}$.
\end{remark}

\newcommand{\calA}{\mathcal{A}}%
\newcommand{\calAhat}{\widehat{\calA}}
\newcommand{\conv}{\operatorname{cnv}}%
\newcommand{\calAconv}{\calA^{\conv}}
\begin{remark}\label{ref:twoConventions:rem}
    Consider a formal deformation $(A_t, \circ{}, +)$. This data can be formulated equivalently by saying that we have
    a $\kkt$-algebra $\calAhat$, which is a free $\kkt$-module, together with an isomorphism $\calAhat/(t)\simeq A$.
    The passage from $(A_t, \circ{}, +)$ to $\calAhat$ amounts to taking $\calAhat := A[\![t]\!]$ with associated multiplication $\circ{}$.
    \begin{enumerate}
        \item If $(A_t, \circ{}, +)$ is of polynomial type, then we can even form a $\kk[t]$-algebra $\calA$, free as a $\kk[t]$-module,
            together with an isomorphism $\calA/(t)\simeq A$. Again, $\calA$ is equal to $A[t]$ with the associated multiplication $\circ{}$.
        \item If $\kk = \mathbb{C}$, we can form $\calAconv$ a
            $\Ct{r}$-algebra, free as a $\Ct{r}$-module, with multiplication
            induced by $\circ{}$. For a complex number $s$ with $|s| < r$, the
            quotient algebra $\calAconv/(t-s)$ is isomorphic to $A_{t=s}$
            defined in Remark~\ref{ref:fibre}.
    \end{enumerate}
    It may seem superfluous to introduce all these subrings of $A[\![t]\!]$, however they are necessary for the proof of
    Theorem~\ref{ref:BeauvilleLaszlo:thm} and, more generally, they are natural to consider when we want to form quotient rings such as $A_{t=s}$.
\end{remark}

For a fixed algebras $A$ and $B$ we would like to have a notion of a deformation of $A$ to $B$.
This is quite straightforward to define for convergent or polynomial type deformations, but it is
a little less intuitive in general, since we cannot evaluate $\kkt$ on $t = s\in \kk \setminus\left\{ 0\right\}$.

We recall two notions: strongly flat deformation (introduced in~\cite[p.1]{Wemyss2}) and various flavours of deformations, present classically throughout algebraic geometry (but not noncommutative algebra), see for example~\cite{G4}.
 \begin{definition}\label{flat}
     A \emph{deformation} of a $\kk$-algebra $A$ to a $\kk$-algebra $B$ is a formal deformation $\calAhat$ as in Definition~\ref{SWbook} such that
    $\calAhat \otimes_{\kkt} \kktpuis$ is isomorphic to $\ttpuis{B} := B\otimes_{\kk} \kktpuis$, where $\kktpuis$ is the field of Puiseux series.
    \begin{enumerate}
        \item A deformation is \emph{split} if $\calAhat[t^{-1}]$ is isomorphic to $\ttfrac{B}$,
        \item A deformation is \emph{polynomially split} it is of polynomial
            type and the associated $\kk[t]$-algebra $\calA$ satisfies
            $\calA[t^{-1}]\simeq B[t^{\pm 1}]$.
    \end{enumerate}
     Assume $\kk = \mathbb{C}$.
     A \emph{strongly flat deformation} (introduced in~\cite[p.1]{Wemyss2}) of $A$ to $B$ is a convergent formal deformation $(A_t, \circ{}, +)$ of $A$ such that
 $A_{t = s_i}$ is isomorphic to $B$ for some sequence of $s_i\in \mathbb{C}$ converging to zero.
\end{definition}
There are some logical dependencies between the notions in this definition.
An isomorphism $\calAhat[t^{-1}]\simeq \ttfrac{B}$ yields an isomorphism
\[
    \calAhat \otimes_{\kkt} \kktpuis\simeq \calAhat[t^{-1}]\otimes_{\kktfrac} \kktpuis\simeq \ttfrac{B}\otimes_{\kktfrac} \kktpuis \simeq  \ttpuis{B},
\]
so being a split deformation implies being a deformation. Similarly, being a polynomially split deformation implies being a split deformation.
At this point it is not clear that a strongly flat deformation is a deformation. We will prove this eventually,
see Theorem~\ref{J1}.

\begin{lemma}[Adding roots of $t$]\label{ref:finiteLevel:lem}
    Let $\calAhat$ be a deformation of $A$ to $B$. Then there is an integer $n\geq 1$ such that
    if we put $\calAhat' := \frac{\calAhat[u]}{u^n - t}$, then $\calAhat'[u^{-1}]\simeq B(\!(u)\!)$ as $\kk(\!(u)\!)$-algebras.
    In particular, if there is a deformation of $A$ to $B$, then there is a split deformation of $A$ to $B$.
\end{lemma}
\begin{proof}
    Consider the composition $f\colon \calAhat \into \calAhat\otimes_{\kkt} \kktpuis \to \ttpuis{B}$, where the first map
    is the canonical one, that is, it maps $a\in \calAhat$ to $a\otimes 1$. Fix elements $a_1, \ldots ,a_d$ such that
    $\calAhat = \kkt a_1 \oplus \ldots \oplus \kkt a_d$. Fix a basis $b_1, \ldots ,b_d$ of the $\kk$-linear space $B$.
    For every $i=1, \ldots ,d$, the element $f(a_i)$ can be written as $\sum_{j=1}^d b_j g_{i,j}$ for some $g_{i,j}\in \kktpuis$.
    By~\eqref{eq:Puiseux}, we can fix a large enough $n$ so that $g_{i,j}\in \kk(\!(t^{1/n})\!)$ for every $i,j=1, \ldots,d$. Define
    $\calAhat'$ as in the statement of the Lemma.
    Observe that $\kk(\!(t^{1/n})\!) \simeq  \kk(\!(u)\!)$.
    We define an isomorphism $\calAhat' \to B(\!(u)\!)$ which sends every $a_i$ to $\sum_{j=1}^d b_j g_{i,j}\in B(\!(u)\!)$.
    The $\kk[\![u]\!]$-algebra $\calAhat'$ is a free $\kk[\![u]\!]$-module and $\calAhat'/(u)\simeq \calAhat/(t)\simeq A$, so $\calAhat'$
    is the required deformation.
\end{proof}

   \begin{example}\label{ex:parabola}
       Take $A = \mathbb{C}[x]/(x^2)$ and $B = \mathbb{C} \times \mathbb{C}$.
       Take $A_t := \mathbb{C}[\![t]\!][x]/(x^2 - t)$. It can be interpreted as a convergent formal deformation (in fact,
       a polynomial one) as follows: we pick a basis $1$, $x$ of $A$ and
       consider the multiplication on $\calAhat = A[\![t]\!]$ given by
       \[
           x\circ{} x := t\cdot 1.
       \]
       It is also a deformation, because
       \[
           \calAhat\otimes_{\kkt} \kktpuis\simeq \frac{\kktpuis[x]}{(x^2 - t)}.
       \]
       The element $t$ in the field $K = \kktpuis$ is a square: $t = (t^{1/2})^2$, so by Chinese remainder theorem, we have
       \[
           \frac{K[x]}{(x^2 - t)} = \frac{K[x]}{(x - t^{1/2})\cdot (x + t^{1/2})}\simeq K \times K
       \]
       which is isomorphic to $B\otimes_{\kk} K$. This shows that $\calAhat$ is a deformation.
       Applying Lemma~\ref{ref:finiteLevel:lem} to this deformation, we see that we
       case take $n = 2$ and $u = t^{1/2}$. However, some $n > 1$ is necessary
       in
       Lemma~\ref{ref:finiteLevel:lem} even in this special case: we have that
       $\calAhat[t^{-1}] = \frac{\kktfrac[x]}{(x^2 - t)}\simeq \kk(\!(t^{1/2})\!)$ is a field, so
       $\calAhat[t^{-1}]$ is not isomorphic to $\kktfrac\times \kktfrac$.
   \end{example}

    \subsection{Moduli space of finite-dimensional algebras}

    The moduli space of finite-dimensional algebras is a classical object, studied 
    for example in~\cite{Artin__Azumaya, Procesi, Voigt, Gabriel, poonen_moduli_space}.
    Let $V = \kk \oplus V'$, where $V' = \kk^{\oplus d-1}$ is a fixed
    $d-1$-dimensional $\kk$-vector space. In this section we contemplate the
    easy observation that multiplication on $V$ can be encoded in terms of
    multiplication constants. Let $e_1 = (1, 0,  \ldots ,0), e_2 = (0, 1,
    0, \ldots ,0), \ldots , e_d$ be a basis of $V$.

    We introduce a tuple of $d^3$ variables $(\lambda_{ij}^k)_{1\leq i,j,k\leq
    d}$ and define the \emph{associativity equations}
    as equations, for every $i,j,l,m$.
    \[
        \sum_{k=1}^d \lambda_{ij}^{k}\lambda_{kl}^m - \sum_{k=1}^d
        \lambda_{ik}^m\lambda_{jl}^k = 0
    \]
    and the \emph{unit equations} for every $j, k$ by
    \begin{equation}\label{eq:unitEquations}
        \lambda_{j1}^k = \lambda_{1j}^k = \begin{cases}
            1 & \mbox{ if } j = k\\
            0 & \mbox{otherwise.}
        \end{cases}
    \end{equation}
    Let us fix $\lambda_{*1}^{*}$ and $\lambda_{1*}^*$ as in~\eqref{eq:unitEquations},
    and we define the \emph{moduli space of $d$-dimensional algebras} as the (affine scheme corresponding
    to) the commutative quotient algebra
        \[
            \cM_d = \frac{\kk[\lambda_{ij}^k\ |\ 2\leq i,j\leq d,1\leq k\leq d]}{\left( \sum_{k} \lambda_{ij}^{k}\lambda_{kl}^m - \sum_{k}
            \lambda_{ik}^m\lambda_{jl}^k\ |\ 1\leq i,j,l,m\leq d
           \right)}.
    \]
    \begin{proposition}\label{prop:associativity}
        Let $*\colon V \times V\to V$ be a $\kk$-bilinear operation.
        Write
        \[
            e_i * e_j = \sum_{k} s_{ij}^k e_k,
        \]
        for every $1\leq i,j,k\leq d$.
        Then $*$ defines an associative multiplication if and only if
        $[s_{ij}^k]$ satisfies associativity equations.
        Similarly, $e_1\in V$ is a unit for $*$ if and only if $[s_{ij}^k]$
        satisfies the unit equations.

        Therefore, the set of associative and unital (with unit $e_1$)
        operations $*$ is in bijection with $\kk$-algebra homomorphisms
        $\cM_d\to \kk$.
    \end{proposition}
    \begin{proof}
        Quite tautological.
        For indices $1\leq i, j\leq d$ we have
        \begin{align*}
            (e_i * e_j) * e_l &= \sum_{k} s_{ij}^k e_k*e_l = \sum_m\sum_{k}
            s_{ij}^ks_{kl}^m e_m\\
            e_i * (e_j * e_l) &= \sum_{k} e_i * e_k s_{jl}^k=\sum_m\sum_k
            s_{ik}^m s_{jl}^k e_m.
        \end{align*}
        Therefore, $*$ is an associative multiplication if and only if the
        tuple $[s_{ij}^k]$ satisfies the associativity equations. Similarly for $e_1$
        being a unit.
    \end{proof}

    More importantly, Proposition~\ref{prop:associativity} generalises to coefficients in rings other than $\kk$.
    \begin{proposition}\label{prop:associativityRelative}
        Let $R$ be a commutative $\kk$-algebra. Let $V_R :=
        V\otimes_{\kk} R$ be a free $R$-module of rank $d$.
        Let $*\colon V_R \times V_R\to V_R$ be a $R$-bilinear operation.
        Write
        \[
            e_i * e_j = \sum_{k} s_{ij}^k e_k,
        \]
        for every $1\leq i,j,k\leq d$ where $s_{ij}^k\in R$.
        Then $*$ defines an associative multiplication if and only if
        $[s_{ij}^k]$ satisfies associativity equations.
        Similarly, $e_1\in V$ is a unit for $*$ if and only if $[s_{ij}^k]$
        satisfies the unit equations.

        Therefore, the set of associative and unital (with unit $e_1$)
        operations $*$ on $V_R$ is in bijection with $\kk$-algebra homomorphisms
        $\cM_d\to R$.
    \end{proposition}

    \begin{example}
        Let $\calAhat$ be a formal deformation as in Remark~\ref{ref:twoConventions:rem}.
        Let $d = \dim_{\kk} \calAhat/(t)$.
        We can identify $\calAhat$ with $V[\![t]\!]$ and hence obtain a homomorphism $\cM_d \to \calAhat$.
        Similarly for other types of deformations.
    \end{example}

    \subsection{Zariski topology on the moduli space}

        The algebra $\cM_d$ is commutative, hence comes with its spectrum
        \[
            \Spec \cM_d = \left\{ \pp \subseteq \cM_d\ |\ \pp\mbox{ prime ideal}
        \right\}\]
        and $\kk$-point spectrum
        \[
            \Speck \cM_d = \left\{ \mm \subseteq \cM_d\ |\ \kk\to
                \cM_d/\mm\mbox{ bijective} \right\}\subseteq \Spec \cM_d.
        \]
        As a set, the $\kk$-point spectrum $\Speck \cM_d$ is a subset of
        \[
            \Speck \kk[\lambda_{ij}^k\ |\ 2\leq i,j\leq d,1\leq k\leq d] \simeq
            \kk^{d(d-1)^2}, 
        \]
        hence
        $\Speck \cM_d$ is a fancy name for the set of all possible structure
        constants.
        In contrast, the set $\Spec \cM_d$ is very complicated and impossible
        to describe in general.
        Both $\Spec \cM_d$ and $\Speck \cM_d$ come with Zariski topology, which is
        not easy to describe explicitly. We omit the definition and
        instead give the most important consequence.

        \begin{proposition}\label{ref:Yoneda:prop}
            Let $f\colon \cM_d\to R$ be any homomorphism of commutative
            $\kk$-algebras. The operation $f^*\colon \qq \mapsto f^{-1}(\qq)$ induces
            \emph{continuous} maps $f^*$
            \[
                \begin{tikzcd}
                    \Speck R\ar[d, hook]\ar[r, "f^*"] & \Speck \cM_d\ar[d,
                    hook]\\
                    \Spec R\ar[r, "f^*"] &\Spec \cM_d
                \end{tikzcd}
            \]
        \end{proposition}

        The proposition is very useful, as for specific $R$ we understand well
        $\Spec R$ and $\Speck R$ both as sets and as topological spaces.
        \begin{example}
            Take $R = \kk[t]$. Then $\Speck R = \kk$, where the bijection sends
            $s\in \kk$ to the ideal $(t - s)$. The topology is cofinite, that
            is, the only closed sets are the whole space and all finite
            subsets. The set $\Spec R$ is slightly larger: it contains the
            zero ideal $(0)$.
        \end{example}

        \begin{example}\label{ex:SpecSeries}
            Take $R = \kkt$. Then $\Spec R = \{(t), (0)\}$. The topology has
            closed sets $\emptyset,\ \{(t)\},\ \Spec R$. The point $(t)$ is closed  in $\Spec R$ and 
            called the \emph{special point}, while the point $(0)$ is open in $\Spec R$ and called
            the \emph{generic point}.
        \end{example}

        \begin{example}\label{ex:SpecConvergent}
            Fix a number $r > 0$. The spectrum $\Spec_{\mathbb{C}} \Ct{r}$
            contains the open disc $\Delta_r := \left\{ z\in \mathbb{C}\ |\ |z|
            < r \right\}$. A point $s\in \Delta_r$ corresponds to the
            evaluation map $\Ct{r} \to \mathbb{C}$ that sends a series
            $\sum_{i\geq 0} \lambda_i t^i$ to $\sum_{i\geq 0} \lambda_i s^i$.
            Consider a Zariski-closed subset of $Z \subseteq\Spec_{\mathbb{C}} \Ct{r}$. It is either the whole
            $\Spec_{\mathbb{C}} \Ct{r}$ or $Z$ is contained in the vanishing set of a nonzero element $f\in \Ct{r}$.
            In the latter case, $Z\cap \Delta_r\subseteq \left\{ z\in \Delta_r\ |\ f(z) = 0 \right\}$ and $Z\cap \Delta_r$ is closed
            in the Euclidean topology. Since $f$ is analytic on
            $\Delta_r$, we see that $Z\cap \Delta_r$ has no accumulation
            points.
        \end{example}

        For $\pp \in \Spec \cM_d$, the ring $\cM_d/\pp$ is an integral domain,
        hence it has a field of fractions $\kappa(\pp) := \Frac(\cM_d/\pp)$
        and an associated homomorphism $\cM_d\to \kappa(\pp)$, which yields a
        $\kappa(\pp)$-algebra structure on $V_{\kappa(\pp)}$.

    \subsection{Comparing various notions of deformations}

    \newcommand{\kappabar}{\overline{\kappa}}%
    \newcommand{\GL}{\operatorname{GL}}%

        Consider a finite-dimensional $\kk$-algebra $A$, not necessarily commutative, with $d = \dim_{\kk} A$.
        Consider also a field $\kk\subseteq\kappa$ and a homomorphism $f\colon \cM_d\to
        \kappa$.
        By Proposition~\ref{ref:Yoneda:prop}, to have $f$ is the same as to fix an associative algebra
        structure $(V_{\kappa}, *)$
        on $V_{\kappa}$.
        We would like to have a way of saying that $(V_{\kappa}, *)$ lies in
        the orbit of $A$, or in other words, that $(V_{\kappa}, *)$ and $A$
        are isomorphic. As stated, this makes no sense, as $\kappa$ might be
        much larger than $\kk$.

        This is amended by tensoring, similarly as we did in Definition~\ref{flat}.
        Let $\kappabar$ be the algebraic closure of $\kappa$.
        We say that $f$ \emph{is a point of the orbit of $A$} if there is an isomorphism of
        $\kappabar$-algebras $(V_{\kappabar}, *)$ and $A\otimes_{\kk} \kappabar$.

        \begin{theorem}\label{thm:gabriel}
            Consider the set $O_A$ of points $\pp\in \Spec \cM_d$ such that $\cM_d\to
            \kappa(\pp)$ is a point of the orbit of $A$. This is a locally closed
            subset of $\Spec \cM_d$ (which means that there is an Zariski-open subset
            $U\subseteq \cM_d$ such that $O_A \subseteq U$ is Zariski-closed).

            If $A$ is semisimple, then $O_A$ is open.
        \end{theorem}
        \begin{proof}
            The group $\GL(V)$ acts on $\cM_d$, and hence on $\Spec \cM_d$. 
            Fix any point $[(V, \circ)]\in \Spec_{\kk} \cM_d$ in the orbit of $A$ and
            consider the map $\GL(V) \to \Spec \cM_d$ given by $g\mapsto g\cdot (V, \circ{})$.
            The image of this map is locally closed by general properties of algebraic group actions
            \cite[Proposition 1.11]{Brion}. A point $\cM_d\to \kappa$ with algebra $(V_{\kappa}, *)$ belongs to the image if and only
            if after passing to the algebraic closure $\kappabar$, we have a matrix $g\in \GL(V_{\kappabar})$
            that proves isomorphism of $(V_{\kappabar}, *)$ and $(V_{\kappabar}, \circ{})$. This concludes
            the general part of the argument.

            The fact that $O_A$ is open for $A$ semisimple follows from~\cite[Corollary~2.5]{Gabriel}.
        \end{proof}

        \begin{definition}\label{def:deformationTopological}
            Let $A$, $B$ be two $\kk$-algebra structures on $V$. Let $f_A,
            f_B\colon \cM_d\to \kk$ be the two corresponding homomorphisms and
            $f_A, f_B\in \Speck \cM_d$ be the two corresponding points. We say that
            \emph{$A$ topologically deforms to $B$} if
            $f_A$ lies in the Zariski-closure of the orbit of $f_B$.
        \end{definition}

        The following transitivity result is nearly obvious on the topological
        side, but its algebraic counterpart (in the setting of Definition~\ref{flat}) is highly
        nontrivial; we believe that any proof of it would employ topology, perhaps in a hidden way. Below
        (Proposition~\ref{ref:degenerationsAlgebraically:prop},
        Theorem~\ref{J1}) we prove that both sides agree, hence we obtain
        transitivity on the algebraic side; this is the whole rationale for
        introducing the topological one.
        \begin{proposition}[Transitivity of deformations]\label{ref:transitivity:prop}
            In the setting of Definition~\ref{def:deformationTopological}, the following hold
            \begin{enumerate}
                \item $A$ topologically deforms to $B$ if and only if the orbit $O_A$ lies in the closure of the orbit
                    $O_B$. In particular, Definition~\ref{def:deformationTopological} does not depend
                    on the choice of $f_A$, $f_B$.
                \item if $A$ topologically deforms to $B$ and $B$ topologically deforms to $C$, then $A$ topologically deforms to $C$.
            \end{enumerate}
        \end{proposition}
        \begin{proof}
            Assume that $A$ deforms to $B$ in the sense of Definition~\ref{def:deformationTopological}.
            Consider the Zariski-closure $\overline{O_A}$. It is Zariski-closed, $\GL(V)$ acts on it and it contains the point
            $f_A$. Hence it also contains $O_A = \GL(V) \cdot f_A$.

            Assume that $A$ deforms to $B$ and $B$ deforms to $C$. This means that $\overline{O_B}\supset O_A$ and $\overline{O_C}\supset O_B$,
            so that
            \[
                \overline{O_C} = \overline{\overline{O_C}} \supset \overline{O_B}\supset O_A,
            \]
            so $A$ deforms to $C$.
        \end{proof}

        The following is known~\cite{Chouhy}, we include it for completeness.
        \begin{proposition}[notions of deformation agree]\label{ref:degenerationsAlgebraically:prop}
            Fix two $\kk$-algebra structures $A$, $B$ on $V$.
            Then the following are equivalent
            \begin{enumerate}
                \item $A$ topologically deforms to $B$ in the sense of Definition~\ref{def:deformationTopological},
                \item there is a deformation of $A$ to $B$ in the sense of Definition~\ref{flat}.
            \end{enumerate}
        \end{proposition}
        \begin{proof}
            Given a deformation $\calAhat$, fix a linear isomorphism $V \simeq A$.
            We get an isomorphism $A_t\simeq V[\![t]\!]$ of $\kkt$-modules and hence
            a morphism $f\colon\Spec \kkt \to \Spec \cM_d$, which maps the
            generic point $(0)$, see Example~\ref{ex:SpecSeries}, to the orbit
            of $B$, while the special point $(t)$ to the orbit of $A$.
            Therefore we have that
            \[
                O_A \ni f((t)) \in f\left(\overline{(0)}\right) =
                f\left(\overline{f^{-1}(O_B)}\right) \subseteq
                f\left(f^{-1}(\overline{O_B})\right)\subseteq O_B
            \]
            and so $A$ topologically deforms to $B$.

            Assume now that $A$ topologically deforms to $B$, so that $O_A \subset \overline{O_B}$. Pick a point $f_A\in O_A$.
            By \cite[Lemma~3.10]{jabu_jelisiejew_smoothability}, there is a homomorphism $\cM_d\to \kkt$ and the
            associated continuous map such that
            the special point of $\Spec \kkt$ maps to $f_A$ and the generic point of $\Spec \kkt$ maps to $O_B$.
            The homomorphism $\cM_d\to \kkt$ is the same as a formal deformation of $A$.
        \end{proof}
        
        For $\kk = \mathbb{C}$ we have the following additional result.
        \begin{proposition}\label{ref:stronglyFlat:prop}
            Let $\kk = \mathbb{C}$ and consider a strongly flat deformation $(A_t, \circ{}, +)$ of $A$ to $B$.
            It yields an associated map $\Delta_r \into \Spec \Ct{r} \to \Spec \cM_d$.
            The image of $\Spec \Ct{r}$ intersects the orbit of $B$ in a dense open subset, so that $A$ topologically deforms to $B$.
        \end{proposition}
        
        \begin{proof}
            The orbit $O_B$ is locally-closed by Theorem~\ref{thm:gabriel}, hence the intersection
            $\Spec \Ct{r}$ is locally-closed, that is, this intersection is equal to $Z \subseteq U \subseteq \Spec\Ct{r}$,
            where $U\subseteq \Spec \Ct{r}$ is open, and $Z \subseteq U$ is closed.
            We would like to show that $Z = U$. By the assumption on strongly flat deformations, $Z\subseteq U$ contains points $s_i$
            arbitrary close to zero, so zero is an accumulation point of $Z$.
            If $Z\neq U$, then, by Example~\ref{ex:SpecConvergent}, the subset
            $Z$ has no accumulation points. This concludes the proof.
        \end{proof}

   \subsection{From $\kkt$ to $\kk[t]$ by descent}

        \newcommand{\calB}{\mathcal{B}}%

            This section contain the key comparison result for various notions of deformation introduced in Definition~\ref{flat}.
            By Lemma~\ref{ref:finiteLevel:lem} we already know that the existence of a deformation forces the existence of a split deformation.
            In this section we obtain a \emph{polynomially} split deformation
            from a split deformation.

            The method is a Beauville-Laszlo descent-like argument.
             Beauville-Laszlo descent is introduced in~\cite{Beauville_Laszlo}
            and a nice presentation in a more general setup is found in~\cite{Banerjee}.
            The question about possible passing from $\kkt$ to $\kk[t]$ is persistent throughout
            deformation theory. For modules it was proven, in a completely different way, in~\cite{Zwara, Yoshino}.

            To clarify the argument, we recall some results about certain
            $\kk[t]$- and $\kkt$-modules. Recall that we use the notation
            $\kk[t^{\pm1}]$ for $\kk[t, t^{-1}]$.
            \begin{enumerate}
                \item Tensoring $(-)\otimes_{\kk[t]} \kk[t^{\pm1}]$ and
                    $(-)\otimes_{\kkt} \kktfrac$ amounts to inverting $t$. In
                    particular, we have $\kkt \otimes_{\kk[t]}
                    \kk[t^{\pm 1}] \simeq \kktfrac$.
                \item Tensoring $(-)\otimes_{\kk[t]} \kkt$, for a finitely
                    generated $\kk[t]$-module, amounts to replacing it by a
                    $\kkt$-module. This is an exact functor, in particular, it preserves
                    kernels and intersections. For a finitely generated
                    $\kkt$-module $M$, we have $M \otimes_{\kk[t]} \kkt$. This
                    fails in general without assuming finite generation.
                    (The key part of Theorem~\ref{ref:BeauvilleLaszlo:thm}
                    below is verifying that it still holds for certain modules
                    which are not obviously finitely generated.)
            \end{enumerate}

            \newcommand{\ihat}{\widehat{i}}%

            \begin{theorem}[Beauville-Laszlo descent or faithfully-flat
                descent]\label{ref:BeauvilleLaszlo:thm}
                Let $\calAhat$ be a split deformation of $A$ to $B$ over
                $\kkt$ and fix an isomorphism $\ihat\colon
                \ttfrac{B}\to\calAhat[t^{-1}]$.
                Let $\calB := B[t^{\pm 1}] \cap \ihat^{-1}(\calAhat)$, where
                the intersection takes place in $\ttfrac{B}
                =\ihat^{-1}(\calAhat[t^{-1}])$. Then
                $\calB$ is $\kk[t]$-algebra which is a polynomially split
                deformation of $A$ to $B$.
            \end{theorem}

            \begin{proof}
                First, the subset $\calB$ is an intersection of subalgebras,
                hence is a subalgebra of $\ttfrac{B}$. The existing maps can
                be summarized in the following diagram
                \[
                    \begin{tikzcd}
                        0 \ar[r] & \calB \ar[r]\ar[d, hook] & B[t^{\pm 1}]
                        \ar[r]\ar[d, hook] &
                        \frac{B[t^{\pm 1}]}{\calB} \ar[r]\ar[d, hook] & 0\\
                        0 \ar[r] & \ihat^{-1}(\calAhat) \ar[r]\ar[d, "\simeq",
                        "\ihat"'] & \ttfrac{B} \ar[r]\ar[d, "\simeq", "\ihat"'] &
                        \frac{\ttfrac{B}}{\ihat^{-1}(\calAhat)} \ar[r]\ar[d,
                        "\simeq"] & 0\\
                        0 \ar[r] & \calAhat \ar[r] & \calAhat[t^{-1}] \ar[r] &
                        \frac{\calAhat[t^{-1}]}{\calAhat} \ar[r] & 0
                    \end{tikzcd}
                \]
                Since inverting an element $t$ is an exact functor, the
                algebra $\calB[t^{-1}]$ is the intersection of $B[t^{\pm 1}]$
                and $\ihat^{-1}(\calAhat[t^{-1}]) = \ttfrac{B}$ in
                $\ttfrac{B}$, so
                $\calB[t^{-1}] = B[t^{\pm 1}]$, as required.

                As a $\kk[t]$-module, the quotient
                $\ttfrac{B}/\ihat^{-1}(\calAhat)$ is isomorphic to
                $(\kktfrac/\kkt)^{\oplus d}$. In particular, the homomorphism
                $\ttfrac{B}/\ihat^{-1}(\calAhat) \to
                (\ttfrac{B}/\ihat^{-1}(\calAhat))\otimes_{\kk[t]} \kkt$ is an
                isomorphism. This implies that the homomorphism
                \[
                    B[t^{\pm 1}] \otimes_{\kk[t]} \kkt \to
                    \ttfrac{B}/\ihat^{-1}(\calAhat) \otimes_{\kk[t]} \kkt
                \]
                is surjective. Also the homomorphism
                \[
                    B[t^{\pm 1}]  \to
                    \ttfrac{B}/\ihat^{-1}(\calAhat)[t^{-1}] = 0
                \]
                is clearly surjective. Since $\kk[t^{\pm 1}] \oplus \kkt$ is a
                faithfully flat $\kk[t]$-module, it follows that
                $B[t^{\pm 1}]\to
                \ttfrac{B}/\ihat^{-1}(\calAhat)$ is surjective.
                The $\kk[t]$-module $\calB$ is the kernel of this map,
                that it fits into the sequence
                \[
                    0\to \calB\to B[t^{\pm 1}] \to
                    \frac{\ttfrac{B}}{\ihat^{-1}(\calAhat)}\to 0.
                \]
                Applying the exact functor $(-)\otimes_{\kk[t]} \kkt$ to this sequence, we
                get
                \[
                    0\to \calB\otimes_{\kk[t]} \kkt\to \ttfrac{B} \to
                    \frac{\ttfrac{B}}{\ihat^{-1}(\calAhat)}\to 0,
                \]
                so that $\calB\otimes_{\kk[t]} \kkt  \simeq \calAhat$.
                As a consequence, we obtain that
                \[
                    \frac{\calB}{t\cdot \calB}  \simeq \calB
                    \otimes_{\kk[t]} \frac{\kk[t]}{(t)}  \simeq \calB
                    \otimes_{\kk[t]} \kkt\otimes_{\kkt}
                    \frac{\kkt}{(t)}  \simeq \calAhat \otimes_{\kkt}
                    \frac{\kkt}{(t)}  \simeq \frac{\calAhat}{t\cdot \calAhat}
                     \simeq A.
                \]

                As another consequence, both $\calB\otimes_{\kk[t]} \kkt$ and
                $\calB\otimes_{\kk[t]} \kk[t^{\pm1}]$ are finitely generated
                free modules over $\kkt$ and $\kk[t^{\pm1}]$, respectively.
                Since $\kkt\oplus \kk[t^{\pm 1}]$ is a faithfully flat
                $\kk[t]$-module, it follows from faithfully flat descent (and
                from $\kk[t]$ being a principal ideal domain), that
                $\calB$ is a finitely generated free $\kk[t]$-module.
            \end{proof}

            The following theorem proves equivalence of all notions of deformations given about.
            It is at this point easy to prove, so we call it a theorem mostly to
            its significant applications. The ideas behind Theorem~\ref{J1} are considered
            relatively standard in the deformation theory community within algebraic geometry. However,
            it seems that they are by no means standard in the noncommutative algebra community, which
            is our target audience. Here, the implication $\eqref{it:defEq7}\implies\eqref{it:defEq6}$
            was proved recently in~\cite[Proposition~30]{Wemyss2}, but the other implications seem new, especially
            the polynomially split deformations.

            \begin{theorem}\label{J1}
                Fix a vector space $V$ of dimension $d$.
                The following are equivalent for (abstract) algebras $B$ and $A$ with $d = \dim_{\kk} A = \dim_{\kk} B$:
                \begin{enumerate}
                    \item\label{it:defEq1} there is a deformation of $A$ to $B$, as in Definition~\ref{flat},
                    \item\label{it:defEq2} there is a split deformation of $A$ to $B$, as in Definition~\ref{flat},
                    \item\label{it:defEq3} there is a polynomially split deformation of $A$ to $B$, as in Definition~\ref{flat},
                    \item\label{it:defEq4}
                        the orbit of $O_{A}$ lies in the Zariski-closure of the orbit of $O_B$,
                    \item\label{it:defEq5}
                        $A$ deforms topologically to $B$, as in Definition~\ref{def:deformationTopological}.
                \end{enumerate}
                If these equivalent conditions hold, we say that $A$ \emph{deforms} to $B$. If $A$ deforms to $B$
                and $B$ deforms to $C$, then $A$ deforms to $C$.

                Moreover, if $\kk = \mathbb{C}$ then these conditions are equivalent to the following
                \begin{enumerate}\setcounter{enumi}{5}
                    \item \label{it:defEq6} there exists a strongly flat deformation of $A$ to $B$ as in
                        Definition~\ref{flat}.
                \end{enumerate}
                Additionally, if $\kk = \mathbb{C}$ and $B$ is semisimple, then they are equivalent to the following condition, see Definition~\ref{SWbook}
                \begin{enumerate}\setcounter{enumi}{6}
                    \item \label{it:defEq7} there exists a deformation of $A$ convergent in radius $r$ with at least
                        one $s\in \mathbb{C}$ with $|s|<r$ such that $A|_{t=s}$ is isomorphic to $B$.
                \end{enumerate}
            \end{theorem}

            \begin{proof}
                The implications
                $\eqref{it:defEq3}\implies\eqref{it:defEq2}$ and
                $\eqref{it:defEq2}\implies\eqref{it:defEq1}$
                follows directly from definitions, see discussion after Definition~\ref{flat}.
                The implication $\eqref{it:defEq1}\implies\eqref{it:defEq3}$
                follows from Lemma~\ref{ref:finiteLevel:lem} and Theorem~\ref{ref:BeauvilleLaszlo:thm}.
                The equivalence
                of~\eqref{it:defEq4} and~\eqref{it:defEq5} follows from
                Proposition~\ref{ref:transitivity:prop}, while the equivalence
                of~\eqref{it:defEq1} and~\eqref{it:defEq4} follows from
                Proposition~\ref{ref:degenerationsAlgebraically:prop}.
                The final transitivity of deformations follows from Proposition~\ref{ref:transitivity:prop}.

                When $\kk = \mathbb{C}$, the
                implication $\eqref{it:defEq3}\implies\eqref{it:defEq6}$ follows from
                the inclusion $\mathbb{C}[t] \into \Ct{r}$. The implication
                $\eqref{it:defEq6}\implies\eqref{it:defEq5}$ follows from Proposition~\ref{ref:stronglyFlat:prop}.
                When $B$ is semisimple, Theorem~\ref{thm:gabriel} shows that its orbit is Zariski-open, so
                the argument of Proposition~\ref{ref:stronglyFlat:prop} applies and gives $\eqref{it:defEq7}\implies\eqref{it:defEq5}$.
                The implication $\eqref{it:defEq6}\implies\eqref{it:defEq7}$ is formal.
            \end{proof}

        \section{Contraction algebras}\label{sec:contraction}

    In this section, we investigate contraction algebras. We use their classification obtained in~\cite{GV} and work mostly with type D.
    We start by giving explicit generating relations and Gr{\"o}bner bases for type A and D contraction algebras.
    Next, we show that each contraction algebra of type D can be (flatly) deformed to at most one semisimple algebra.
    Finally, we show that such a deformation indeed exists. The results seem to be new even if one assumes full knowledge
    of the geometric side. There, Toda~\cite{Toda} constructed families of contraction algebras, but it is not known if these are flat.

    \subsection{Description of contraction algebras of type A and D}
    Contraction algebras come in three types: A, D, E, see~\cite{GV} for explanation of their connection to simply laced Dynkin diagrams of type A, D, E.
    In the current article we tackle types A and D, leaving E for future work. The type A is very simple, so
    the main part concerns type D.

    \begin{notation}
        By $\kkxy$ we will denote the free noncommutative algebra on two variables $x$ and $y$.
        By $\kkxyhat$ we will denote the completion of $\kkxy$ along the maximal ideal $(x, y)$.
    \end{notation}
    By definition, a contraction algebra is given by a quotient
    \[
        \frac{\kkxyhat}{(\partial_x f, \partial_y f)},
    \]
    where $f\in \kkxyhat$ is a certain series and $\partial_x f$, $\partial_y f$ are cyclic derivatives~\cite{GV}.
    Below we give these explicitly.

    \begin{definition}\label{ref:contractionTypes:def}
        Define the following \emph{contraction algebras}~\cite[Table~1]{GV} given by $\kkxyhat/(\partial_x f, \partial_y f)$ of types A and D.
        \begin{center}
            \begin{tabular}{c c c c}
                type & $f$ & $\partial_x f$ & $\partial_y f$\\
                \toprule 
                $A_n$ & $\frac{1}{2}x^2 + \frac{1}{n}y^n$ & $x$ & $y^{n-1}$\\
                $D_{n,m}$ & $\frac{1}{2n}x^{2n} + \frac{1}{2m-1}x^{2m-1} + xy^2$ & $x^{2n-1} + x^{2m-2} + y^2$ & $xy+yx$\\
                $D_{n,\infty}$ & $\frac{1}{2n}x^{2n} + xy^2$ & $x^{2n-1} + y^2$ & $xy + yx$
            \end{tabular}

            The indexes $n, m$ can be any integers $n, m\geq 2$.
        \end{center}
        We say that an algebra is a \emph{contraction algebra of type $D$} if it is isomorphic to one of the algebras
        $D_{n,m}$ or $D_{n, \infty}$.
    \end{definition}

    Definition~\ref{ref:contractionTypes:def} does not immediately yield
    relations of contraction algebras, because the completion $\kkxyhat$ can
    yield new generating relations.
    The work~\cite[\S6]{GV} contains explicit description in terms of generators and relations. We reprove these
    here for the convenience of the reader.
    We also fix our notation for standard bases, which are power-series analogues of Gr{\"o}bner bases, see~\cite{GH}.
    We consider the deglex monomial order with $\deg(y) \gg \deg(x)$. Below we give both a set of generating relations,
    already after completion, and additional relations (only one each time) that appear in the Gr{\"o}bner basis (GB).

    \begin{proposition}[{\cite[\S6]{GV}}]\label{ref:Generators:prop}
        Let $n, m\geq 2$. We have the following description of contraction algebras from Definition~\ref{ref:contractionTypes:def} as quotients of $\kkxy$.
        \begin{center}
            \begin{tabular}{c c c c c}
                type & generating relations & GB added rels. & dim. & dim. centre\\
                $A_n$ & $x, y^{n-1}$ & none & $n$ & $n$\\ 
                $D_{n,m}$, $m \leq n$ & $xy+yx, x^{2n-1} + x^{2m-2} + y^2, x^{2n+2m-3}$ & $x^{2n-1}y$ & $4n+2m-4$ & $n+2m-1$\\
                $D_{n,m}$, $m > n$ & $xy+yx, x^{2n-1} + x^{2m-2} + y^2, x^{4n-2}$ &$x^{2n-1}y$ & $6n-3$ & $3n$\\
                $D_{n,\infty}$ & $xy+yx, x^{2n-1} + y^2, x^{4n-2}$ & $x^{2n-1}y$ & $6n-3$ & $3n$
            \end{tabular}
        \end{center}
        Additionally, we define $D_{n,1}$ and $D_{1, m}$, $D_{1, \infty}$ to be given by the same generating relations as in the table above.
    \end{proposition}
    \begin{proof}
        For type $A$, we have $\kkxyhat/(x, y^{n-1})$ is already commutative and isomorphic to
        \[
            \kk[\![y]\!]/(y^{n-1})\simeq \kk[y]/(y^{n-1}).
        \]
        Let us pass to type $D_{n,m}$ with $m\leq n$. 

        The ideal of $\kkxy$ generated by
        $ x^{2n-1}+x^{2m-2}+y^{2}, xy+yx$ contains $x^{2n-1}y$, hence $x^{2n-1}y^{2}$, and  hence also $x^{2n-1}(x^{2n-1}+x^{2m-2})$.
        Since $m\leq n$, we have $2n-1 > 2m-2$.
        Modulo every power of $(x, y)$, the relation $x^{2n-1}(x^{2n-1}+x^{2m-2})$ yields $x^{2n-1}\cdot x^{2m-2}$. Therefore, in the completion
        we have also the relation $x^{2n-1+(2m-2)}$.
        The ideal
        \[
            I = \left( xy+yx,\ x^{2n-1} + x^{2m-2} + y^2,\ x^{2n+2m-3} \right)\subseteq \kkxy
        \]
        contains a power $(x, y)^N$ for $N \gg 0$, hence there are no more elements in the completion of $I$.
        We verify directly that
        \[y^{2}+x^{2m-1}+x^{2n-1},\ yx+xy,\ x^{2n-1}y,\ x^{2n+2m-3}\]
        form a standard basis for the ideal $I$. The quotient $\kkxy/I$ as a vector space has a basis $1, x, \cdots , x^{2n+2m -4}$, $y, xy, \cdots x^{2n-2}y$ and this has dimension $4n+2m-4$. Notice that 
        \[x^{2n+2m-4}\in I,\quad\mbox{but}\quad x^{2n+2m-3}\notin I.\]

 The centre of this algebra has dimension $n+2m-1$ and is spanned as vector space by $x^{2i}$ for $i=0, 1,2, \ldots n+m-2$ ($n+m-1$ elements) and $x^{2n-1+2i}$ for $i=0, 1, \cdots , m-2$ ($m-1$ elements), and $x^{2n-2}y$ (one element).
 
        The consideration for type $D_{n,m}$ with $m\geq n$ and for $D_{n,\infty}$ are completely analogous.
        The only difference is that for $m\geq n$, we have $2n - 1 < 2m-2$, so we obtain $x^{4n-2}$ in the ideal $I\subseteq \kkxy$ defining the algebra.
        This implies that $\dim_{\kk} \kkxy/I = 6n-3$ and that the centre has
        dimension $3n$ and a basis $x^{2i}$ for $i=0, 1,2, \ldots 2n-2$ ($2n-1$
        elements) and $x^{2n+2i-1}$ for $i=0, 1, \cdots , n-1$ ($n$ elements),
        and $x^{2n-2}y$ (one element).
    \end{proof}

    \begin{remark}\label{ref:Dinfty:rem}
        Since $x^{4n-2}$ is a relation of $D_{n,m}$ for $m > n$, we have $D_{n,m} = D_{n, \infty}$ for every $m \geq 2n$.
    \end{remark}

    \subsection{Obstructions to deformations to semisimple algebras}    

        Here we formulate several semicontinuity results that restrict possible semisimple
        algebras that admit deformations from contraction algebras. This method is classical, present
        already in~\cite[Proposition~2.7]{Gabriel} or \cite[Proposition~23, Corollary~27]{Wemyss2}.

        We will use the following formulation. Fix noncommutative polynomials
        \[
            w_{1}(x_1, \ldots ,x_k), \ldots ,w_l(x_1, \ldots ,x_k) \in \kk\langle x_1, \ldots ,x_k\rangle.
        \]
        For an algebra $A$ and elements $a_1, \ldots ,a_k\in A$ by
        $w_{i}(a_{1},  \ldots , a_{k})$ we denote the evaluation of polynomial
        $w_i(x_{1},  \ldots , x_{k})$ at $x_{1}=a_{1},  \ldots , x_{k}=a_{k}$, and with
        the multiplication from $A$. By $\spann{w_{\bullet}(a_1, \ldots ,a_k)}$ we denote
        the linear space over $\kk$ spanned by $(w_{i}(a_1, \ldots ,a_k))_{i=1, \ldots ,l}$.
        Similarly, for algebra $B$ with elements $b_1, \ldots ,b_k\in B$, by $w_{i}(b_{1},  \ldots , b_{k})$ we denote the evaluation of polynomial
        $w_i(x_{1},  \ldots , x_{k})$ at $x_{1}=b_{1},  \ldots , x_{k}=b_{k}$, and with
        the multiplication from $B$. By $\spann{w_{\bullet}(b_1, \ldots ,b_k)}$ we denote
        the linear space over $\kk$ spanned by $(w_{i}(a_1, \ldots ,a_k))_{i=1, \ldots ,l}$.
        \begin{proposition}[semicontinuity, {\cite[Proposition~23]{Wemyss2}}]\label{ref:Agata:prop}
            Let $A$, $B$ be finite-dimensional $\kk$-algebras such that $A$ deforms
            to $B$. Fix noncommutative polynomials $w_1, \ldots ,w_l\in \kk\langle x_1, \ldots ,x_k \rangle $ and elements $a_1, \ldots ,a_k\in A$.
            Then there exist elements $b_1, \ldots ,b_k\in B$ such that
            \[
                \dim_{\kk} \spann{w_{\bullet}(a_1, \ldots ,a_k)} \leq \dim_{\kk} \spann{w_{\bullet}(b_1, \ldots ,b_k)}.
            \]
        \end{proposition}
        \begin{proof}
            We give a proof for completeness.
            We can enlarge the field $\kk$ to an algebraically closed field.
            Let $\calA$ be a $\kk[t]$-algebra which is a split polynomial deformation of $A$ to $B$,
            see Theorem~\ref{J1}, so that $\calA[t^{-1}]\simeq B[t^{\pm 1}]$ and $\calA/(t)\simeq A$.
            Let $\alpha_1, \ldots ,\alpha_k\in \calA$ be any elements which map to $a_1, \ldots ,a_k$ modulo $t$.
            Let $v_1, \ldots ,v_d\in \calA$ be such that
            \[
                \calA = \kk[t] v_1 \oplus  \ldots \oplus \kk[t] v_d.
            \]
            Write each element $w_i(\alpha_1, \ldots ,\alpha_k)\in \calA$ as a combination $\sum_{j=1}^d w_{ij} v_j$, where $w_{ij}\in \kk[t]$.
            Consider the matrix
            \[
                W := [w_{ij}]_{i=1, \ldots , l, j=1, \ldots ,d}\in \mathbb{M}_{l \times d}(\kk[t]).
            \]
            The reduction of $W$ modulo $t$ is the matrix of coefficients of $w_i(a_1, \ldots ,a_k)$ in a basis $v_1 \mod t, \ldots ,v_d \mod t$.
            Similarly, for every $\lambda\in \kk \setminus \{0\}$, the
            reduction of $W$ modulo $(t-\lambda)$ is the matrix of coefficients
            of $w_i(b_1, \ldots ,b_k)$ in a basis $v_1 \mod t, \ldots ,v_d \mod
            t$ of $B$, where $b_1, \ldots ,b_k\in B\simeq \calA/(t-\lambda)$ are images of $\alpha_1, \ldots ,\alpha_k$.
            Therefore, the statement in the theorem reduces to the claim that
            \def\rk{\operatorname{rk}}%
            \[
                \rk(W_{t=0}) \leq \sup_{\lambda\in \kk \setminus \{0\}} \rk (W_{t = \lambda})
            \]
            but this is true, since the rank of a matrix can only decrease at $t = 0$, as it is given by polynomial conditions
            (nonvanishing of minors).
        \end{proof}

        \begin{corollary}\label{ref:ObstructionSmallMatrices}
            Let $R$ be a contraction algebra of type $D$. Then $R$ does not deform to any algebra $R'$ such that
            $R'$ surjects onto $\mathbb{M}_{n}(\kk)$ for $n\geq 3$.
        \end{corollary}
        \begin{proof}
            Consider the elements $x^i,\ x^{i}y$ for $i=0,1, \ldots, \dim_{\kk}
            R$. They span $R$, hence by semicontinuity
            (Proposition~\ref{ref:Agata:prop}), there exist elements $x', y'\in R$ such that $(x')^i$, $(x')^iy'$
            for $i=0,1 \ldots , \dim_{\kk} R'$ span $R'$. Since $\mathbb{M}_{n}(\kk)$ is an image of $R$', there are matrices
            $X'$, $Y'$ such that $\mathbb{M}_{n}(\kk)$ is spanned by
            \[
                \left\{ (X')^i\ |\ i=0,1, \ldots  \right\} \cup \left\{ (X')^iY'\ |\ i=0,1, \ldots  \right\}.
            \]
            The matrix $X'$ satisfies its characteristic polynomial, so $1, X', \ldots, (X')^n$ are linearly dependent and the
            elements above span a space of dimension at most $2n$. We get that $2n\geq n^2$, so $n\leq 2$.
        \end{proof}

\begin{theorem}\label{ref:obstructionsMain:thm}
    Let $N$ be a contraction algebra of type A or D. Then $N$ can be flatly deformed to at most one semisimple algebra. More precisely, for each type
    the only possible deformation is as follows
    \begin{center}
        \begin{tabular}{c c}
            type & can only deform to\\
            $A_n$ & $\kk^{\times n}$\\
            $D_{n,m}$, $m \leq n$ & $\mathbb{M}_{2}(\kk)^{\times n-1} \times \kk^{2m}$\\
            $D_{n,m}$, $m > n$ & $\mathbb{M}_{2}(\kk)^{\times n-1} \times \kk^{2n+1}$\\
            $D_{n,\infty}$ & $\mathbb{M}_{2}(\kk)^{\times n-1} \times \kk^{2n+1}$
        \end{tabular}
    \end{center}
    The algebras in both column have centres of same dimension.
\end{theorem}

\begin{proof}
    By Corollary~\ref{ref:ObstructionSmallMatrices}, the semisimple algebras that could possibly occur are of the form
    \[
        A_{k,l} = \mathbb{M}_{2}({\kk})^{\times k}\times {\kk}^{\times l}
    \]
    for some $k, l\geq 0$.
    The algebra $A_{k,l}$ has dimension  $4k+l$, and its centre has dimension $k+l$. We will use the following two facts:
    \begin{enumerate}
        \item \emph{Fact 1.}
            If $N$ deforms to $A_{k,l}$, then the dimension of the centre of $N$ is greater or equal to the dimension of the centre of $A_{k,l}$,
            see for example~\cite[Proposition~2.7]{Gabriel} or perform an argument analogous to Proposition~\ref{ref:Agata:prop}.
        \item \emph{Fact 2.}
            If $c\in N$ is an nilpotent element such that $c^{d-1} \neq 0$, then $N$ does not deform to $A_{k,l}$ with $2k+l < d$.

            Indeed, by nilpotency, we have that $1,c, \ldots ,c^{d-1}$ are linearly independent in $N$, so Proposition~\ref{ref:Agata:prop} given
            an element in $c' \in A_{k,l}$, not necessarily nilpotent, whose
            powers $1,c', \ldots ,(c')^{d-1}$ are linearly independent. But a
            matrix satisfies its characteristic polynomial, so the span of powers of $c'$ in every $\mathbb{M}_{2}(\kk)$ is at most $2$.
            In total this gives a space of dimension $2k+l$. If $2k+l < d$, this gives a contradiction.
    \end{enumerate}
    \emph{Case $0$.} For type A,  Fact~2 implies that $A_{k,l}$ is also generated by one element, hence it is commutative.

    \emph{Case $1$.} Consider first the cases $D_{n,m}$ with $m > n$, which include, by Remark~\ref{ref:Dinfty:rem}, also $D_{n, \infty}$. Here, the dimension of $D_{n,m}$ is $6n-3$, the centre of $D_{n,m}$ is $3n$-dimensional and $x^{4n-3}\neq 0$. In both cases we obtain conditions
    \begin{align*}
        6n-3 &=4k+l\\    
        4n-2 &\leq 2k+l\\
        k+l &\leq 3n
    \end{align*}
    We obtain that $l$ is odd and that
    \[
        12n - 6 = 3\cdot (4n-2) \leq 6k + 3l = 3(2k + l) = 4k+l + 2(k + l) \leq 6n-3 + 6n = 12n-3.
    \]
    The number $3(2k+l)$ is odd and divisible by three, hence $3(2k+l) = 12n - 3$, so $k = n-1$, $l = 2n+1$.

    \emph{Case $2$.} The case $D_{n,m}$ with $n\leq m$ is similar. The two facts imply the following inequalities
    \begin{align*}
        4n+2m-4 &=4k+l\\    
        2n+2m-3 &\leq 2k+l\\
        k+l &\leq n+2m-1
    \end{align*}
    We obtain
    \begin{align*}
        3n - 3 = (4n+2m-4) - (n+2m-1) &\leq (4k + l) - (k+l) = 3k\\
        2n- 1 = (4n+2m-4) - (2n+2m-3) &\geq (4k + l) - (2k+l) = 2k
    \end{align*}
    so that $n-1 \leq k < n$, hence $k = n-1$, $l = 2m$.
\end{proof} 

\subsection{Existence of deformations of $D_{m,n}, D_{n, \infty}$ to semisimple algebras}

To provide the required deformations, we will make heavy use of transitivity of deformations proven in Theorem~\ref{J1}.
We will also use Remark~\ref{ref:Dinfty:rem} and consider only the $D_{n,m}$ case.

We proceed inductively. 
The base of the induction is formed by cases $D_{1, \infty}$ and $D_{n,1}$ as defined in Proposition~\ref{ref:Generators:prop}.
We will use the following observation.

\begin{remark}\label{ref:matrices:rem}
    The algebra $\kkxy/(xy+yx, x^2 - 1, y^2-1)$ is isomorphic to $\mathbb{M}_2(\kk)$. The isomorphism
    sends $x, y$ to 
    \[
        \begin{pmatrix}
            0 & 1\\
            1 & 0
        \end{pmatrix}\qquad
        \begin{pmatrix}
            1 & 0\\
            0 & -1
        \end{pmatrix},
    \]
    respectively.
\end{remark}

\begin{lemma}\label{ref:D1infty:lem}
    The algebra $D_{1, \infty}$ deforms to $\kk \times \kk \times \kk$.
\end{lemma}
\begin{proof}
    The algebra is a quotient of $\kkxy$ by the ideal $(xy+yx, x + y^2, xy, x^{2})$. The second equation allows
    to substitute $x := -y^2$ and we obtain a quotient $\kk[y]/(y^3)$. The required deformation is given by $\kk[y, t]/(y\cdot (y-t)\cdot (y+t))$.
\end{proof}

\begin{lemma}\label{ref:Dn1:lem}
    For $n\geq 1$, the algebra $D_{n, 1}$ deforms to $\mathbb{M}_2(\kk)^{\times n-1} \times \kk\times \kk$.
\end{lemma}
\begin{proof}
    The algebra $D_{n, 1}$ is a quotient of $\kkxy$ by the ideal $(xy+yx, 1 + y^2, x^{2n-1})$.
    By multiplying $y$ with a fourth root of unity, we change this ideal to $(xy+yx, y^2 - 1, x^{2n-1})$.
    Consider the algebra $\calA = \kkxy[t]/(xy+yx, y^2 - 1, x(x^{2n-2} - t^{2n-2}))$.
    It is a free $\kk[t]$-module of rank $4n-2$ and $\calA/(t) = D_{n,1}$.

    We will prove that $\calA$ is a polynomially split deformation of $D_{n,1}$
    to $\mathbb{M}_2(\kk)^{\times n-1} \times \kk\times \kk$.

    Let $\omega_1, \ldots ,\omega_{n-1}$ be solutions of $s^{n-1} = 1$
    in $\kk$. We have
    \[
        x^{2n-2} - t^{2n-2} = \prod_{i=1}^{n-1}\left( x^2 - \omega_i t^2 \right).
    \]
    Since $\kk$ has characteristic zero, the elements $x$ and $x^2 - \omega_i
    t^2$ for $i=1,
    \ldots , n-1$ are pairwise coprime and normal in $\calA[t^{-1}]$.
    By repeatedly using the Chinese Remainder Theorem, we obtain that
    \[
        \calA[t^{-1}]\simeq \frac{\calA[t^{-1}]}{(x)} \times \prod_{i=1}^{n-1} \frac{\calA[t^{-1}]}{(x^2 - \omega_i t^2)}
    \]
    The algebra $\calA[t^{-1}]/(x)$ is a quotient of $\kk[x, t^{\pm 1}]$ by $(y^2 - 1)$, hence is isomorphic to $\kk[t^{\pm 1}]\times \kk[t^{\pm 1}]$,
    as necessary. Fix $i$ and consider the algebra
    \[
        \calA_i := \frac{\calA[t^{-1}]}{(x^2 - \omega_i t^2)} \simeq \frac{\kkxy[t^{\pm 1}]}{(xy+yx,\ y^2 - 1,\ x^2 - \omega_i t^2)}.
    \]
    By substituting $x := xt\sqrt{\omega_i}$, we see that this quotient is isomorphic to the algebra
    $\kkxy/(xy + yx, y^2 - 1, x^2 - 1)[t^{\pm 1}]$. By Remark~\ref{ref:matrices:rem}, we know that this is $\mathbb{M}_2(\kk)[t^{\pm 1}]$.
\end{proof}

The induction step requires us to introduce the following auxiliary algebra
\begin{equation}\label{eq:A2}
    A_2 := \frac{\kkxy}{\left( xy + yx,\ y^2,\ (x^2 - 1)y,\ (x^2-1)^2 \right)}
\end{equation}
 \begin{lemma}\label{jeden}
     The algebra $A_2$ is $6$-dimensional, with $3$-dimensional centre and it can be deformed to 
     $\mathbb{M}_{2}(\kk)\times \kk \times \kk$.
\end{lemma}

\begin{proof}
    Consider the $\kk[t]$-algebra $A_2(t) := \kkxy[t]/I$, where $I$ is generated by $xy + yx$, $y^2$, $(x^2 - 1)y$, $(x^2-1)\cdot (x^2 - t)$.
    This is a free $\kk[t]$-module of rank $6$. For every $\lambda\in \kk \setminus \{1, 0\}$, consider the algebra
    $A_2(\lambda) := A_2(t) / (t-\lambda)$. Explicitly speaking, it is the algebra
    \[
        A_2(\lambda) = \frac{\kkxy}{\left( xy + yx,\ y^2,\ (x^2 - 1)y,\ (x^2-1)\cdot (x^2 - \lambda) \right)}.
    \]
    Since $\lambda\neq 1$, the central elements $x^2 - 1$, $x^2 - \lambda$ together generate the unit ideal. By the Chinese remainder theorem
    \[
        A_2(\lambda) \simeq \frac{A_2(\lambda)}{(x^2 - 1)} \times \frac{A_2(\lambda)}{(x^2 - \lambda)}.
    \]
    In the quotient $A_2(\lambda)/(x^2-\lambda)$, the image of $y$ is zero, so $A_2(\lambda) \simeq \frac{\kk[x]}{(x^2 - \lambda)} \simeq \kk \times \kk$,
    since $\kk$ is algebraically closed and $\lambda\neq 0$. Let us analyse
    \[
        A' := \frac{A_2(\lambda)}{(x^2 - 1)}\simeq \frac{\kkxy}{(xy + yx,\ y^2,\ x^2 - 1)}.
    \]
    Take
    \[
        A'(t) := \frac{\kkxy[t]}{(xy + yx,\ y^2 - t,\ x^2 - 1)}.
    \]
    This is a free $\kk[t]$-module of rank $4$ with basis $1$, $x$, $y$, $xy$. For every $t = \mu\in \kk\setminus \{0\}$, we  
    can choose $\nu\in \kk$ such that $\nu^2 = \mu$. After a change $y\mapsto y\cdot \nu$, we have
    \[
        \frac{A'(t)}{(t - \mu)} = \frac{\kkxy[t]}{(xy + yx,\ y^2 - \nu^2,\ x^2 - 1)}\simeq \frac{\kkxy[t]}{(xy + yx,\ y^2 - 1,\ x^2 - 1)}.
    \]
    The rightmost algebra is the matrix algebra $\mathbb{M}_2(\kk)$ by Remark~\ref{ref:matrices:rem}. Therefore, the algebra $A'(t)$ yields a deformation of
    $A'$ to $\mathbb{M}_2(\kk)$. Consequently, for every $\lambda\neq 0,1$, the algebra $A_2(\lambda)$ topologically deforms
    to $\mathbb{M}_2(\kk) \times \kk \times \kk$. Thus, the same is true for $\lambda = 1$.
\end{proof}

\begin{proposition}\label{trzy} Let $m,n\geq 2$ be natural numbers. Then the Jacobi algebra  $D_{n,m}$ can be flatly deformed to the Jacobi algebra  $D_{n-1, m-1}\times A_{2}$.    Moreover, the Jacobi algebra $D_{n, \infty }$ can be deformed to $D_{n-1, \infty }\times A_{2}$. 
\end{proposition}

\begin{proof}
    Consider first $m \leq n$.
    By Proposition~\ref{ref:Generators:prop}, the algebra $D_{n,m}$ has standard basis consisting of elements
    \[
        xy+yx,\ x^{2n-1} + x^{2m-2} + y^2,\ x^{2n+2m-3},\ x^{2n-1}y.
    \]
    Consider $\calA := \kkxy[t]/I$, where the ideal $I$ is generated by 
    \[
        xy+yx,\ x^{2(n-1)+2(m-1)-3}\left(x^{2}-t^{2}\right)^{2},\ x^{2(n-1)-1}\left(x^{2}-t^{2}\right)y,\ y^{2}+x^{2(n-1)-1}\left(x^{2}-t^{2}\right)+x^{2(m-1)-2}\left(x^{2}-t^{2}\right).
    \]
    The above relations form a Gr{\"o}bner basis of $\kkxy$ with the deglex monomial order defined above Proposition~\ref{ref:Generators:prop}.
    This implies that $\calA$ is a free $\kk[t]$-module.
    The elements $x^{2(n-1)+2(m-1)-3}$ and $(x^2 - t^2)^2$ are central in $\calA[t^{-1}]$, multiply to zero and are coprime in $\calA[t^{-1}]$, so by the Chinese remainder theorem
    we have
    \[
        \calA[t^{-1}] \simeq \frac{\calA[t^{-1}]}{(x^{2(n-1)+2(m-1)-3})} \times \frac{\calA[t^{-1}]}{(x^2 - t^2)^2}.
    \]
    Let us analyse the two factors.
    For the first factor, in the algebra $\calA[t^{-1}] / (x^{2(n-1)+2(m-1)-3})$ the element $x$ is nilpotent, so $x^2 - t^2$ is invertible.
    Let $y' := y\cdot (x^2-t^2)^{-1}$, so that $y = y'\cdot (x^2 - t^2)$. In variables $x$, $y'$, the algebra $\calA[t^{-1}]$ becomes
    \[
        \frac{\kk\langle x, y'\rangle[t^{\pm1}]}{(xy' + y'x,\ x^{2(n-1)+2(m-1)-3},\ x^{2(n-1)-1}y',\ (y')^2 + x^{2(n-1)-1} + x^{2(m-1)-2})}
    \]
    which is exactly $D_{n,m}[t^{\pm 1}]$.
    For the second factor, recall the algebra $A_2$ from~\eqref{eq:A2}.
    The algebra $\calA[t^{-1}]/(x^2 - t^2)^2$ is given explicitly as
    \[
        \frac{\kkxy[t]}{\left( xy+yx,\ (x^2 - t^2)^2,\ (x^2 - t^2)\cdot y,\ y^2 \right)},
    \]
    which is isomorphic to $A_2[t^{\pm 1}]$ via the map that takes $x$ to $tx$. Putting both factors together, we observe that $\calA$ is
    a polynomially split deformation of $D_{n-1,m-1}\times A_2$ to $D_{n,m}$.

    Consider now $m > n$. The proof in this case is completely analogous, the only difference is that the relation $x^{2n+2m-3}$ is replaced
    by $x^{4n-2}$ which in $\calA$ becomes $x^{4(n-1)-2}\cdot (x^2 - t^2)^2$.
\end{proof}

By using transitivity of deformations (Theorem~\ref{J1}), we obtain the following corollary: 
\begin{theorem}\label{ref:existence:thm}
    Let $B=\mathbb{M}_{2}({\kk}) \times \kk \times \kk$. 
    If $n\geq m$ then the Jacobi Algebra $D_{n,m}$ can be flatly deformed to the algebra $D_{n-m+1, 1}\times B^{\times m-1}$ and to $\mathbb{M}_2(\kk)^{\times n-1}\times \kk^{2m}$.
 If $n < m$, then the Jacobi Algebra $D_{n,m}$ can be flatly deformed to the algebra $D_{1, \infty}\times B^{\times n-1}$ and to $\mathbb{M}_2(\kk)^{\times n-1}\times \kk^{2n+1}$.
\end{theorem}
\begin{proof} 
    By Proposition~\ref{trzy} we get that $D_{n,m}$ can be deformed to $D_{n-1,m-1}\times A_{2}$. By Lemma~\ref{jeden}, the algebra $A_{2}$ can be deformed to 
    $B = \mathbb{M}_{2}(\kk)\times \kk\times \kk$. Therefore, by transitivity of deformations,
 we get that $D_{n,m}$ can be deformed to $D_{n-1,m-1}\times B$.
 
 If $n\geq m$ then the first deformation follows by induction on $n$ since we can decrease $n,m$ to get $m=1$.
 The second result follows from Lemma~\ref{ref:Dn1:lem}.

  If $m>n$, then we can decrease $n,m$ to get $n=1$. By Remark~\ref{ref:Dinfty:rem}, we then have $D_{1,m}\simeq D_{1,\infty}$.
  This yields the first deformation. The second follows from Lemma~\ref{ref:D1infty:lem}.
\end{proof}

{\bf Acknowledgments.} 
 The authors would like to thank  Michael Wemyss for his inspiring questions and for comments about applications of deformations of noncommutative algebras.
 The second author thanks {\O}yvind Solberg for help with understanding the QPA package.
The first-named author is supported by National Science Centre grant 2023/50/E/ST1/0033.
This research of the second author was supported by 
EPSRC programme grant EP/R034826/1.

\end{document}